\documentclass[11pt,english,oneside]{amsart}

\usepackage[breaklinks,colorlinks,linkcolor=blue,citecolor=blue,urlcolor=blue]{hyperref}
\usepackage[T1]{fontenc}
\allowdisplaybreaks[4]
\usepackage[latin9]{inputenc}
\usepackage{geometry}
\usepackage{indentfirst}
\usepackage{amssymb}
\usepackage{color}

\usepackage{booktabs}
\usepackage{array, caption, threeparttable}
\usepackage[font=small,labelfont=bf,labelsep=none]{caption}
\usepackage{graphicx,color}
\usepackage{amsmath, amssymb, graphics}

\usepackage{graphicx}

\makeatletter
\numberwithin{equation}{section} 
\numberwithin{figure}{section} 
  \@ifundefined{theoremstyle}{\usepackage{amsthm}}{}
  \theoremstyle{plain}
  
  \theoremstyle{plain}
  
  \theoremstyle{plain}
  
  \theoremstyle{remark}
  
  \theoremstyle{remark}
  
  \theoremstyle{plain}
  
\usepackage{geometry}

\newcommand{\R}{\mathbb{R}}

\smallskip
\def\<{{\langle }}
\def\>{{\rangle }}

\makeatother

\usepackage{babel}

\smallskip
\def\<{{\langle }}
\def\>{{\rangle }}

\makeatother

\usepackage[mathscr]{eucal}
\usepackage{amssymb}
\usepackage{latexsym}
\usepackage{amsthm}
\theoremstyle{plain}
\newtheorem{theorem}{Theorem}[section]
\newtheorem{proposition}{Proposition}[section]

\newtheorem{remark}{Remark}[section]
\newtheorem{lemma}{Lemma}[section]

\makeatletter
\@addtoreset{equation}{section}

\usepackage{color}

\usepackage{graphicx,color}

\title[Examples of compact embedded convex $\lambda$-hypersurfaces]
{Examples of compact embedded convex $\lambda$-hypersurfaces}
\author{Qing-Ming Cheng, Junqi Lai and  Guoxin Wei}
\address{Qing-Ming Cheng \\ Department of Applied Mathematics, Faculty of Science,
Fukuoka  University, 814-0180, Fukuoka,  Japan, cheng@fukuoka-u.ac.jp}
\address{Junqi Lai \\  School of Mathematical Sciences, South China Normal University,
510631, Guangzhou,  China, 2019021668@m.scnu.edu.cn}
\address{Guoxin Wei \\  School of Mathematical Sciences, South China Normal University,
510631, Guangzhou,  China, weiguoxin@tsinghua.org.cn}

\begin{document}
\maketitle

\begin{abstract}
In the paper, we construct compact embedded convex $\lambda$-hypersurfaces which are diffeomorphic to a sphere and are not isometric to
a  standard  sphere. As the special case of our result, we solve Sun's problem (Int Math Res Not: 11818-11844, 2021). In this sense, one can not expect to have Alexandrov type theorem for $\lambda$-hypersurfaces.
\end{abstract}

\footnotetext{The first author was partially  supported by JSPS Grant-in-Aid for Scientific Research (B):  No.16H03937 and No.22K03303.
The third author was partly supported by NSFC Grant No.12171164, GDUPS (2018).}

\section{Introduction}

 A hypersurface $\Sigma^n \subset \R^{n+1}$ is called a $\lambda$-hypersurface if it satisfies
\begin{equation}\label{0316eq1.1}
     H + \< X, \nu\> = \lambda,
\end{equation}
where $\lambda$ is a constant, $X$ is the position vector, $\nu$ is an inward unit
normal vector and $H$ is the mean curvature. 
The notation of $\lambda$-hypersurfaces were first introduced by Cheng and Wei in \cite{CW} (also see \cite{MR}). Cheng and Wei \cite{CW} proved that $\lambda$-hypersurfaces are critical points of the weighted area functional with respect to weighted
volume-preserving variations. This equation of $\lambda$-hypersurfaces also arises in the study of isoperimetric problems in weighted (Gaussian) Euclidean spaces, which is a long-standing topic studied in various fields in science. $\lambda$-hypersurfaces can also be viewed as stationary solutions to the isoperimetric problem in the Gaussian space. For more information on $\lambda$-hypersurfaces, one can see \cite{CW} and \cite{MR}.

Firstly, we give some examples of $\lambda$-hypersurfaces. It  is well known that  there are several special complete embedded solutions to (\ref{0316eq1.1}):
\begin{itemize}
	\item hyperplanes with a distance of $|\lambda|$ from the origin,
	\item sphere with radius $\frac{-\lambda + \sqrt{\lambda^2 + 4n }}{2}$ centered at origin,
	\item cylinders with an axis through the origin and radius $\frac{-\lambda + \sqrt{\lambda^2 + 4(n - 1) }}{2}$.
\end{itemize}

\noindent
In \cite{CW1}, Cheng and Wei constructed the first nontrivial example of a $\lambda$-hypersurface which is diffeomorphic to $\mathbb S^{n-1} \times \mathbb S^1$ by using techniques similar to Angenent \cite{A}. In \cite{R}, using a similar method to McGrath \cite{M}, Ross constructed a $\lambda$-hypersurface in $\R^{2n+2}$ which is diffeomorphic to $\mathbb S^{n} \times \mathbb S^{n} \times \mathbb S^1$ and exhibits a $SO(n) \times SO(n) $ rotational symmetry. In \cite{LW}, Li and Wei constructed an immersed $S^n$ $\lambda$-hypersurface using a similar method to \cite{D}. It is quite interesting to find other nontrivial examples of $\lambda$-hypersurfaces.

Secondly, we introduce some rigidity results about $\lambda$-hypersurfaces. If $\lambda=0$, $\langle X, \nu\rangle +H=\lambda=0$, then $X:\Sigma^n\to  \mathbb{R}^{n+1}$ is a self-shrinker of  mean curvature flow, which plays an important role for study on singularities of the mean curvature flow. Abresch and Langer \cite{AL} proved that the only $1$-dimensional compact embedded self-shrinker is the circle. Huisken \cite{H1} proved any compact embedded, and mean convex ($H\geq0$) self-shrinkers are spheres. Later, Colding and Minicozzi \cite{CM} generalized Huisken's results to the case of complete self-shrinkers.

\noindent
If $\lambda\neq0$, there are relatively few results about $\lambda$-hypersurfaces. In \cite{CW}, Cheng and Wei proved that generalized cylinders $\mathbb S^{m} \times \mathbb R^{n-m}$, $0\leq m\leq n$ are the  only complete embedded $\lambda$-hypersurfaces with polynomial volume growth in $\mathbb R^{n+1}$ if
$H-\lambda\geq0$ and $\lambda (f_3(H-\lambda)-S)\geq0$, where $f_3=\sum_{i,j,k=1}^n h_{ij}h_{jk} h_{ki}$, $S=\sum_{i,j=1}^n h_{ij}^2$, $h_{ij}$ denotes the components of the second fundamental form. This classification result generalizes the result of Huisken \cite{H1} and Colding and Minicozzi \cite{CM}.

\noindent For $\lambda>0$, in \cite{Hei}, Heilman proved that convex $n$-dimensional  $\lambda$-hypersurfaces are generalized cylinders if $\lambda>0$, which generalized the rigidity result of Colding and Minicozzi \cite{CM} to $\lambda$-hypersurfaces with $\lambda>0$.

\noindent The case $\lambda<0$ is much more complicated. In \cite{G1}, by using the explicit expressions of the derivatives of the principal curvatures at the non-umbilical points of the surface, Guang obtained that any strictly mean convex $2$-dimensional $\lambda$-hypersurfaces are convex if $\lambda\leq0$. Later, by using the maximum principle, Lee \cite{L1} showed that any compact embedded and mean convex $n$-dimensional $\lambda$-hypersurfaces are convex if $\lambda\leq0$. In \cite{C}, Chang constructed $1$-dimensional $\lambda$-curves in $\mathbb R^2$, which are not circles. These surprising examples show that $\lambda$-surfaces behave very differently when $\lambda<0$ compared to the case $\lambda>0$. New techniques must be introduced to study this phenomenon. We imagine these examples may carry meaningful information in probability theory (also see \cite{S}).

In \cite{S}, Sun developed the compactness theorem for $\lambda$-surface in $\mathbb R^{3}$ with uniform $\lambda$ and genus. As the application of the compactness theorem, he also showed a rigidity theorem for convex $\lambda$-surfaces. In the same paper, he \cite{S} proposed the problem that constructing a compact convex $\lambda$-surface which is not a sphere (see Question 4.0.4. on page 25, \cite{S}).

In this paper, motivated by \cite{D,DK,LW,S}, we construct nontrivial embedded $\lambda$-hypersurfaces which are diffeomorphic to $\mathbb S^{n}$
and they are not isometric to a standard sphere.  As the special case of our result, that is, $n=2$, we solve Sun's problem. In fact, we obtain the following theorem:
\begin{theorem}\label{0316thm1.1}
For  $n \ge 2$ and $ -\frac{2}{\sqrt{n+2}} < \lambda <0$,  there exists an  embedded convex $\lambda$-hypersurface $\Sigma^n \subset \R^{n+1}$ which is diffeomorphic to $\mathbb{S}^n$ and is not  isometric to a standard sphere.
\end{theorem}

\begin{remark}
It is well-known that for compact embedded hypersurfaces with constant mean curvature in $\R^{n+1}$,
Alexandrov theorem holds, that is, a compact embedded hypersurface with constant mean curvature in $\R^{n+1}$ is isometric to a round sphere.
But for $\lambda$-hypersurfaces, one can not expect to have Alexandrov type theorem for $\lambda$-hypersurfaces according to  the above theorem 1.1.
\end{remark}

\begin{remark}

For self-shrinkers, there is a well-known conjecture asserts that the round sphere  should
be the only embedded self-shrinker which is diffeomorphic to a sphere. Brendle \cite{B} proved the above conjecture for $2$-dimension self-shrinker. For the higher dimensional  self-shrinker, the conjecture is still open. But for $\lambda$-hypersurfaces, we can construct compact embedded $\lambda$-hypersurface which is diffeomorphic to a sphere and is not isometric to a round sphere.
\end{remark}

\begin{remark}
For $-\frac{2}{\sqrt{3}}<\lambda<0$, Chang \cite{C} proved that there exists a compact embedded $\lambda$-curve with $2$-symmetry in $\R^{2}$, which is not a circle.
\end{remark}

%

\section{Preliminaries}\label{prel}
Let $SO(n)$ denote the special orthogonal group and act on $\R^{n+1} = \left\lbrace (x,y): x\in\R, y\in\R^n\right\rbrace $ in the usual way, then we can identify the space of orbits $\R^{n+1}/SO(n)$ with the half plane $\mathbb{H} = \left\lbrace (x,r)\in\R^2: x\in\R, r \ge 0\right\rbrace$ under the projection (see \cite{R})
\[
 \Pi(x,y) = (x,|y|) = (x,r).
\]
If a hypersurface $\Sigma$ is invariant under the action $SO(n)$, then the projection $\Pi(\Sigma)$ will give us a profile curve in the half plane, which can be parametrized by Euclidean arc length and write as $\gamma(s) = (x(s), r(s))$. Conversely, if we have a curve $\gamma(s) = (x(s), r(s)),\ \ s \in (a, b)$ parametrized by Euclidean arc length in the half plane, then we can reconstruct the hypersurface by
\begin{gather}
    	 { X : (a, b) \times S^{n-1}(1) \hookrightarrow \R^{n+1} }, \notag \\
    	 (s,\alpha) \mapsto (x(s), r(s)\alpha). \label{0316eq2.1}
\end{gather}
Let
\begin{equation}\label{0316eq2.2}
	\nu = (-\dot{r}, \dot{x}\,\alpha),
\end{equation}
where the dot denotes taking derivative with respect to arc length $s$. A direct calculation shows that $\nu$ is an inward unit normal vector for the hypersurface. Then we can calculate that the principal curvatures of hypersurface (see \cite{CW1,DD}): 
\begin{equation}\label{0426eq2.3}
\aligned
\kappa_i & = - \frac{\dot{x}}{r}, \ \ \ \ i=1,\,2,\,\dots,\,n-1, \\
\kappa_{n} &= \dot{x}\,\ddot{r}-\ddot{x}\,\dot{r}.
\endaligned
\end{equation}
Hence the mean curvature vector equals to
\begin{equation}\label{0316eq2.3}
     \overrightarrow{H}=H\nu,\ \ \ \ {\mbox here}\ \  H =\sum_{i=1}^n\kappa_i=  \dot{x}\,\ddot{r}-\ddot{x}\,\dot{r} - (n-1)\frac{\dot{x}}{r},
\end{equation}
and then, by (\ref{0316eq2.1}), (\ref{0316eq2.2}) and (\ref{0316eq2.3}), equation (\ref{0316eq1.1}) reduces to (see also \cite{CW1,DK,R})
\begin{equation}\label{0316eq2.4}
     \dot{x}\,\ddot{r} - \ddot{x}\,\dot{r} = (\frac{n-1}{r} - r )\dot{x} +   x\,\dot{r} + \lambda,
\end{equation}
where $(\dot{x})^2 + (\dot{r})^2 = 1$. Let $\theta(s)$ denote the angles between the tangent vectors of the profile curve and $x$-axis, then (\ref{0316eq2.4}) can be written as the following system of differential equation:
\begin{equation}\label{0316eq2.5}
 \left\lbrace
    \begin{aligned}
	\dot{x} &= \cos\theta,\\
	\dot{r} &= \sin\theta, \\
	\dot{\theta} &=  (\frac{n-1}{r} - r )\cos\theta +   x\,\sin\theta + \lambda.
	\end{aligned}
\right.
\end{equation}
Let $P$ denote the projection from $\mathbb{H} \times \R$ to $\mathbb{H}$. Obviously, if $\gamma(s)$ is a solution of (\ref{0316eq2.5}), then the curve $P(\gamma(s))$ will generate a $\lambda$-hypersurface by (\ref{0316eq2.1}). If we can, by this way, find a curve that starts and ends on the $x$-axis and is perpendicular to the  $x$-axis at both ends, then  we obtain an embedded  $\lambda$-hypersurface, which is diffeomorphic to a sphere. Hence,  Theorem \ref{0316thm1.1} will be proved. By some symmetry of (\ref{0316eq2.5}), a curve that starts perpendicularly on the $x$-axis and ends perpendicularly on the $r$-axis will also solve the problem. This paper's main goal is to find the latter.
\\
Letting $(x_0, r_0, \theta_0)$ be a point in $\{(x, r, \theta): x\in\R, r > 0, \theta \in \R \}$, by an existence and uniqueness theorem of the solutions for first order ordinary differential equations, there is a unique solution $\Gamma(x_0, r_0, \theta_0)(s)$ to (\ref{0316eq2.5}) satisfying initial conditions $\Gamma(x_0, r_0, \theta_0)(0) = (x_0, r_0, \theta_0) $. Moreover, the solution depends smoothly on the initial conditions. Note that (\ref{0316eq2.5}) has a singularity at $r = 0$ due to the term $\frac{1}{r}$. But there is also an existence and uniqueness theorem for (\ref{0316eq2.5})  with initial conditions at singularity and $\theta_0 = \pi/2$. And the previously mentioned smooth dependence  can also extend to singularity smoothly (see \cite{D}).

When the profile curve can be written in the form $(x, u(x))$, by (\ref{0316eq2.5}), the function $u(x)$ satisfies the differential equation
\begin{equation}\label{0316eq2.6}
	\frac{u''}{1 + (u')^2} = x\,u' - u + \frac{n-1}{u} + \lambda\sqrt{1+(u')^2}.
\end{equation}
When the profile curve can be written in the form $(f(r), r)$, by (\ref{0316eq2.5}), the function $f(r)$ satisfies the differential equation
\begin{equation}\label{0316eq2.7}
\frac{f''}{1 + (f')^2} =  ( r - \frac{n-1}{r}  )f' -f - \lambda\sqrt{1+(f')^2}.
\end{equation}
Next we consider equations (\ref{0316eq2.5}) in polar coordinates,
when the profile curve can be written in the form $\rho = \rho(\phi)$, where $\rho = \sqrt{x^2\,+\,r^2 }$ and $\phi\,=\,\arctan(r/x)$, the function $\rho(\phi)$ satisfies the differential equation
\begin{equation}\label{0316eq2.8}
\rho''\,=\,\frac{1}{\rho}\left\lbrace \rho'^2\,+\, (\rho^2\,+\,\rho'^2)\left[n\,-\,\rho^2\,-\,(n\,-\,1)\frac{\rho'}{\rho}\cot\phi\,-\,\lambda\,\sqrt{\rho^2\,+\,\rho'^2} \right] \right\rbrace.
\end{equation}
Note that (\ref{0316eq2.7}) has a solution $f = -\lambda$, which corresponds to hyperplane, (\ref{0316eq2.8}) has a solution $\rho\,=\,\frac{-\lambda\,+\,\sqrt{\lambda^2\,+\,4\,n}}{2}$, which corresponds to the round sphere.
\\
We conclude this section with a lemma about the solutions of (\ref{0316eq2.7}), this lemma can be found in {\cite{LW}}.
\begin{lemma}[\cite{LW}]\label{0316lem2.1}
	Let $f$ be a solution of (\ref{0316eq2.7}) with $f(0) > -\lambda$ and $f'(0) = 0$. Then we have
	 $f'' <0$ and there exists  a point $r_* < \infty$ so that $\lim_{r \rightarrow r_*}f'(r) = -\infty $ and $\lim_{r \rightarrow r_*}f(r) > -\infty $, i.e., $f$ blows-up at $r_*$.
\end{lemma}

\section{Behavior near the known solutions}
In order to solve the problem,  we need to study the behavior near two known solutions first, and  use continuity argument to find the curve that we want to have.
\subsection{Behavior near the plane}
Let $f_{\epsilon}(r)$ be the solution of (\ref{0316eq2.7}) with $f_{\epsilon}(0)= -\lambda + \epsilon$ and $f_{\epsilon}'(0) = 0$ where $'$ denotes $\frac{{\rm d}}{{\rm d}r}$. For $\epsilon > 0$, by Lemma \ref{0316lem2.1}, let $r_*^\epsilon$ denote the point where $f_\epsilon$ blows-up, and let $x_*^\epsilon = f_\epsilon(r_*^\epsilon)$. Let $h_\epsilon(r) = f_\epsilon(r) + \lambda $,  we may write $h_\epsilon(r)$ to $h(r)$, $f_\epsilon(r)$ to $f(r)$, $r_*^\epsilon$ to $r_*$ and $x_*^\epsilon$ to $x_*$ when there is no ambiguity. We introduce the following proposition which gives a description of the behavior near the plane $f_0 = -\lambda$.

\begin{proposition}\label{prop3.1}
	For any fixed $n \ge 2$, $-\min\left\lbrace \frac{23n+4}{28\sqrt{n}}, \frac{25n - 6}{30\sqrt{n}}\right\rbrace  \le \lambda < 0$, there exists $\bar{\epsilon}>0$ so that
	\[
	r_*^\epsilon \ge \sqrt{\log \frac{1}{\sqrt{\pi}\epsilon}},\ \ \ \
	-\frac{30n+4}{\sqrt{\log \frac{1}{\sqrt{\pi}\epsilon}}} - \lambda \le x_*^\epsilon <  -\lambda
	\]
for $\epsilon \in (0, \bar{\epsilon}]$.
\end{proposition}
To prove the proposition, we need several information  on $f_{\epsilon}(r)$ under small $\epsilon$.

\begin{lemma}\label{lem3.1}
	If $\lambda \le 0$, $0 < \epsilon \le \frac{1}{\sqrt{\pi}}$, then $r_* > \sqrt{\log \frac{1}{\sqrt{\pi}\epsilon}}$.
\end{lemma}
\begin{proof}
	
	Since $f'(0) = 0$, $f''(r) < 0$, and $\lim_{r \rightarrow r_*}f'(r) = -\infty$, there exists $r' \in (0, r_*)$ so that $f'(r') = -1$. Then $h'(r')$ also equals to $-1$. For $r \in (0, r')$, we have
	\[
	\aligned
	  \dfrac{{\rm d}}{{\rm d}r}(e^{-r^2}h'(r))
	  &= e^{-r^2}h''(r) - 2re^{-r^2}h'(r)\\
	  &\ge \dfrac{2}{1 + h'(r)^2} e^{-r^2}h''(r) - 2re^{-r^2}h'(r)\\
	  &= 2e^{-r^2}\left[ (r - \dfrac{n - 1}{r})h'(r) - h(r) + \lambda(1 - \sqrt{1 + h'(r)^2}) \right] - 2re^{-r^2}h'(r)\\
	  &\ge -2e^{-r^2}h(r),
	\endaligned
	\]
	where we have used that $\lambda \le 0$ in the last inequality. Integrating from $0 \ {\rm to}\  r'$,
	\[
	   -e^{-(r')^2} \ge -2\int_{0}^{r'}e^{-r^2}h(r){\rm d}r \ge -2\epsilon\int_{0}^{r'}e^{-r^2}{\rm d}r \ge -\sqrt{\pi}\epsilon.
	\]
	Hence,  we obtain $r_* > r' \ge \sqrt{\log \frac{1}{\sqrt{\pi}\epsilon}}$ for $\lambda \le 0$, $0 < \epsilon \le \frac{1}{\sqrt{\pi}}$.
\end{proof}

\begin{lemma}\label{lem3.2}
	For any fixed $n \ge 2$, $\lambda \in \mathbb{R}$, there exists $\epsilon_1 > 0$ so that $f_{\epsilon}(r) = -\lambda$ $($i.e., $h_\epsilon(r) = 0$$)$ has a solution in $[\sqrt{n},\sqrt{2n}]$ for $|\epsilon| \le \epsilon_1$.
\end{lemma}
\begin{proof}
	We adopt an argument from the appendix B of \cite{DK}.
In order to understand the behavior of $f_{\epsilon}(r)$ when $\epsilon$ is close to 0, we study the linearization of the rotational $\lambda$-hypersurface differential equation (\ref{0316eq2.7}) near the plane $f_0(r) = -\lambda$. We define $w$ by
	$$
	w(r) = \dfrac{{\rm d}}{{\rm d}\epsilon}\bigg|_{\epsilon = 0} f_{\epsilon}(r).
	$$	
	Since $f_{\epsilon}(r)$ satisfies  equation
	$$
	\dfrac{ f_{\epsilon}''}{1+ (f_{\epsilon}')^2} = (r-\dfrac{n-1}{r}) f_{\epsilon}' -  f_{\epsilon} - \lambda\sqrt{1+ (f_{\epsilon}')^2},
	$$
	by  differentiating  the above equation with respect to $\epsilon$ and letting  $\epsilon = 0$, we obtain a differential equation for $w$:
	\begin{equation}\label{eq3.1}
	w'' = (r - \dfrac{n-1}{r})w' - w,
	\end{equation}
	with $w(0) = 1$ and $w'(0) = 0$, where we have used $f_{0}'(r) = 0$. We will show that $w(\sqrt{n}) > 0$ and $w(\sqrt{2n}) < 0 $ which lead to the lemma.
	\\
	Defining $\xi = r^2$, the (\ref{eq3.1}) becomes into the following differential equation:
	\begin{equation}\label{eq3.2}
	4\xi \dfrac{{\rm d}^2w}{{\rm d}\xi^2} = 2(\xi - n)\dfrac{{\rm d}w}{{\rm d}\xi} - w
	\end{equation}
	with the initial conditions at $\xi = 0$:
	\[
	w(0) = 1, \ \ \ \dfrac{{\rm d}w}{{\rm d}\xi}(0) = -\dfrac{1}{2n}.
	\]
	The equation (\ref{eq3.2}) is one of the classical differential equations, namely a confluent hypergeometric equation. Up to a dilation of the argument $\xi$, the solutions $w$ used here are called
Kummer functions. Taking derivatives of (\ref{eq3.2}), we have the following second order differential equation:
	\begin{equation}\label{eq3.3}
	4\xi \dfrac{{\rm d}^3w}{{\rm d}\xi^3} = 2(\xi - n-2)\dfrac{{\rm d}^2w}{{\rm d}\xi^2} + \dfrac{{\rm d}w}{{\rm d}\xi}
	\end{equation}
	with the initial conditions at $\xi = 0$:
	\[
	\dfrac{{\rm d}w}{{\rm d}\xi}(0) = -\dfrac{1}{2n}, \ \ \ \dfrac{{\rm d}^2w}{{\rm d}\xi^2}(0) = -\dfrac{1}{4n(n+2)}.
	\]
	We also note that the differential equation (\ref{eq3.3}) for $\frac{{\rm d}w}{{\rm d}\xi}$ satisfies a maximum principle,
	which yields that $\frac{{\rm d}^2w}{{\rm d}\xi^2} < 0$ and then $\frac{{\rm d}w}{{\rm d}\xi} \le \frac{{\rm d}w}{{\rm d}\xi}(0) < 0$.
	Hence $w$ is strictly concave and strictly decreasing on $[0,\infty)$.
	Combining this fact with $w|_{\xi = 0}= 1 $ and $\frac{{\rm d}w}{{\rm d}\xi}(0) = -\frac{1}{2n}$
	we obtain $w|_{\xi = 2n} < 0$. Let $\xi = n$ in (\ref{eq3.2}), we get $w|_{\xi = n} > 0$.
	Now we have proved that $w|_{\xi = n} > 0$ and  $w|_{\xi = 2n} < 0$, i.e.,  $w|_{r = \sqrt{n}} > 0$ and $w|_{r = \sqrt{2n}} < 0 $.
\end{proof}
\begin{lemma}\label{lem3.3}
	Suppose $n \ge 2$, $-\frac{25n-6}{30\sqrt{n}} \le \lambda <0$. 
	If  there exists a point $r_0 \in [\sqrt{n}, \infty)$ so that $h(r_0) = 0$, we have
\[
	h(r) > \dfrac{30n}{r}h'(r)
\]
for $r \in [r_0, r_*)$.
\end{lemma}
\begin{proof}
	 Let $\Phi(r) = \frac{1}{30n}rh(r) - h'(r)$. We want to show that $\Phi(r) > 0$ for $r \in [r_0, r_*)$.
	 If $r = r_0$, then $\Phi(r_0) = - h'(r_0) > 0$.  If $r > r_0$, by a  direct computation, we obtain
	\[
	\aligned
	\Phi(r)
	&= \frac{1}{30n}rh(r) - h'(r)\\
	&= \dfrac{1}{30n}r\int_{r_0}^{r}h'(\xi){\rm d}\xi - h'(r) \\
	&> \dfrac{1}{30n}r(r-r_0)h'(r) - h'(r)\\
	&= \dfrac{1}{30n}h'(r)(r^2 - r_0r - 30n)\\
	&\ge  \dfrac{1}{30n}h'(r)(r^2 - \sqrt{n}r - 30n),\\
	\endaligned	
	\]
	that is, $\Phi(r) > 0$ for $r \le 6\sqrt{n}$.
	\\
	Suppose $\Phi(r) = 0$ for some $r \in [r_0, r_*)$. Then $r > 6\sqrt{n}$ and there exists a point $\bar{r} \in (6\sqrt{n}, r_*)$ so that $\Phi(\bar{r}) = 0$ and $\Phi(r) > 0$ for $r \in  [r_0, \bar{r})$. This implies that $\Phi'(\bar{r}) \le 0$ and $\dfrac{1}{30n}\bar{r}h(\bar{r}) = h'(\bar{r})$. On the other hand, since
	\begin{align*}
	\Phi'(\bar{r})
	&=  \dfrac{1}{30n}h(\bar{r}) + \dfrac{1}{30n}\bar{r}h'(\bar{r}) - h''(\bar{r})  \\
	&\ge \dfrac{1}{30n}h(\bar{r}) + \dfrac{1}{30n}\bar{r}h'(\bar{r}) - \dfrac{h''(\bar{r})}{1 + h'(\bar{r})^2} \\
	&= \dfrac{1}{30n}h(\bar{r}) + \dfrac{1}{30n}\bar{r}h'(\bar{r}) -
	\left[ (\bar{r} - \dfrac{n - 1}{\bar{r}})h'(\bar{r}) - f(\bar{r}) - \lambda\sqrt{1 + h'(\bar{r})^2} \right]  \\
	&= \dfrac{1}{30n}h(\bar{r}) + \dfrac{1}{30n}\bar{r}h'(\bar{r}) -
	\left[ (\bar{r} - \dfrac{n - 1}{\bar{r}})h'(\bar{r}) - h(\bar{r}) + \lambda(1 - \sqrt{1 + h'(\bar{r})^2}) \right]  \\
	&\ge \dfrac{1}{30n}h(\bar{r}) + \dfrac{1}{30n}\bar{r}h'(\bar{r}) -
	\left[ (\bar{r} - \dfrac{n - 1}{\bar{r}})h'(\bar{r}) - h(\bar{r}) + \lambda h'(\bar{r}) \right]  \\
	&= \dfrac{1}{30n}h(\bar{r}) + \dfrac{1}{900n^2}\bar{r}^2h(\bar{r}) -
	\left[\dfrac{1}{30n}\bar{r} (\bar{r} - \dfrac{n - 1}{\bar{r}} + \lambda)h(\bar{r}) - h(\bar{r}) \right]  \\
	&= \dfrac{1}{900n^2}h(\bar{r})\left[ (1 - 30n)\bar{r}^2 - 30n\lambda \bar{r} +930n^2 \right]  \\
	&> 0,
\end{align*}
where we have used that $n \ge 2$, $-\frac{25n-6}{30\sqrt{n}} \le \lambda <0$ and $\bar{r} > 6\sqrt{n}$ in the last inequality,
this contradicts  $\Phi'(\bar{r}) \le 0$. Hence, we   complete  the proof.
\end{proof}
\noindent {\it Proof of Proposition \ref{prop3.1}.}  
The first part of the proposition \ref{prop3.1} has been shown in the lemma \ref{lem3.1}. We know that $r' \ge \sqrt{\log \frac{1}{\sqrt{\pi}\epsilon}}$ for $0 < \epsilon \le \frac{1}{\sqrt{\pi}}$ from  the lemma \ref{lem3.1}. In particular, $r'  \ge \sqrt{\log \frac{1}{\sqrt{\pi}\epsilon}} \ge 7\sqrt{n}$ for $0 < \epsilon \le \frac{1}{\sqrt{\pi}e^{49n}}$. Choosing  $\bar{\epsilon} = \min \left\lbrace \epsilon_1,  \frac{1}{\sqrt{\pi}e^{49n}} \right\rbrace $ and assuming  $0< \epsilon \le \bar{\epsilon}$,
 by the lemma \ref{lem3.2} and lemma \ref{lem3.3}, we have $h(r_*) < h(r') < h(r_0) =0$ (that is, $x_* < -\lambda$) and
\[
  h(r') > -\dfrac{30n}{r'} \ge -\dfrac{30n}{\sqrt{\log \frac{1}{\sqrt{\pi}\epsilon}}}.
\]
We will extend this estimate for $h(r')$ to an estimate for $h(r_*) = x_* + \lambda$. For $r \ge r'$, we have
\[
\aligned
   h''(r)
   &< h'(r)^2\dfrac{h''(r)}{1 + h'(r)^2}\\
   &=   h'(r)^2\left[ (r - \dfrac{n - 1}{r})h'(r) - h(r) + \lambda(1 - \sqrt{1 + h'(r)^2}) \right]\\
   &=   h'(r)^2\left[ (r - \dfrac{n - 1}{r})h'(r) - h(r) + \lambda h'(r) \right]\\
   &<   h'(r)^3(r - \dfrac{31n-1}{r} + \lambda)\\
   &\le \dfrac{1}{4}rh'(r)^3,
\endaligned
\]
where we have used that $n \ge 2,\ r \ge 7\sqrt{n}$ and $-\frac{23n +4}{28\sqrt{n}} \le \lambda <0$ in the last inequality.
\\
Integrating the previous inequality from $r$ to $r_*$, implies
\[
   h'(r)^2 \le \dfrac{4}{r_*^2 - r^2}
\]
for $ r \ge r'$. Since $h'(r) < 0$, we have
\begin{equation}\label{eq3.4}
	h'(r) \ge -\dfrac{2}{\sqrt{r_*^2 - r^2} } \ge -\dfrac{1}{\sqrt{r_* + r'}}\dfrac{2}{\sqrt{r_* - r}}
\end{equation}
for $r \in [r', r_*)$. At $r'$, this tells us that
\[
    -\dfrac{\sqrt{r_* - r'}}{\sqrt{r_* + r'}} \ge -\dfrac{2}{r_* + r'}.
\]
Finally, integrating (\ref{eq3.4}) from $r'$ to $r_*$, we have
\[
   h(r_*) - h(r') \ge -\dfrac{4}{\sqrt{r_* + r'}}\sqrt{r_* - r'},
\]
and therefore
\[
\aligned
    h(r_*)
    &\ge h(r') - \dfrac{4}{\sqrt{r_* + r'}}\sqrt{r_* - r'}\\
    &\ge -\dfrac{30n}{r'} - \dfrac{8}{r_* + r'}\\
    &\ge -\dfrac{30n+4}{r'}\\
    &\ge -\frac{30n+4}{\sqrt{\log \frac{1}{\sqrt{\pi}\epsilon}}}.
\endaligned
\]
Hence $x_* \ge  -\frac{30n+4}{\sqrt{\log \frac{1}{\sqrt{\pi}\epsilon}}} -\lambda$, this competes the proof. \qed

\subsection{Behavior near the round sphere}
As in  the proof of the lemma \ref{lem3.2}, we also make use of  the method from the appendix B of \cite{DK}.
Recall equation (\ref{0316eq2.8}),  if we can write the profile curve in the form $\rho = \rho(\phi)$ in polar coordinates, where $\rho = \sqrt{x^2\,+\,r^2 }$ and $\phi\,=\,\arctan(r/x)$,  then (\ref{0316eq2.8}) tells us that $\rho = \rho(\phi)$ satisfies
\begin{equation}\label{0314eq3.5}
    \rho''\,=\,\frac{1}{\rho}\left\lbrace \rho'^2\,+\, (\rho^2\,+\,\rho'^2)\left[n\,-\,\rho^2\,-\,(n\,-\,1)\frac{\rho'}{\rho}\cot\phi\,-\,\lambda\,\sqrt{\rho^2\,+\,\rho'^2} \right] \right\rbrace. \tag{2.8}
\end{equation}
This ordinary differential equation has a constant solution
\[
 \rho\,=\,\frac{-\lambda\,+\,\sqrt{\lambda^2\,+\,4\,n}}{2}
\]
which corresponds to the round sphere. We note that this equation has a singularity when $\phi =  0$ due to the $\cot\phi$ term. Let $\rho(\phi, \epsilon)$ be the solution to (\ref{0314eq3.5}) with initial conditions  $\rho(0, \epsilon) = \frac{-\lambda + \sqrt{\lambda^2 + 4\,n}}{2} + \epsilon$ and $\frac{{\rm d}\rho}{{\rm d}\phi}(0, \epsilon) = 0$. As is Drugan and Kleene said in \cite{DK}, for $\lambda = 0$, the solution $\rho(\phi, \epsilon)$ is smooth when $\phi\in [0, \pi/2]$ and $\epsilon$ is close to 0. This is also true for  $\lambda \in \R$.
\\
In order to understand the behavior of $\rho(\phi,\epsilon)$ when $\epsilon$ is close to 0, we study the linearization of
the rotational self-shrinker differential equation near the round sphere $\rho(\phi,0) =  \frac{-\lambda + \sqrt{\lambda^2 + 4\,n}}{2}
$. We define $w$ by
\[
   w(\phi) = \frac{\rm d}{{\rm d}\epsilon}\bigg|_{\epsilon = 0}\rho(\phi,\epsilon).
\]
Then $w$ satisfies the (singular) linear differential equation:
\begin{equation}\label{0315eq3.6}
        w''  = -(n\,-\,1)\cot \phi\,w'\,-\,\frac{-\lambda\,+\,\sqrt{\lambda^2\,+\,4\,n}}{2}\,\sqrt{\lambda^2\,+\,4n}\,w
\end{equation}
with $w(0)  = 1$ and $w'(0)=0$. Letting $A(\lambda) = \frac{-\lambda\,+\,\sqrt{\lambda^2\,+\,4\,n}}{2}\,\sqrt{\lambda^2\,+\,4n}$, we may write $A(\lambda)$ for  $A$ when there is no ambiguity. Note that the sign of $A$ is positive for $\lambda \in \R$. In the following lemma, we claim that $w(\pi/2)<0$ and $w'(\pi/2)<0$ which yields that, for $\epsilon < 0$ and $\epsilon$ closed to $0$, $r>\frac{-\lambda\,+\,\sqrt{\lambda^2\,+\,4\,n}}{2}$ and $\pi/2 < \theta < \pi$ at the first point where the profile curve meets $r$-axis.
\begin{lemma}\label{0317lem3.4}
Let $-\frac{2}{\sqrt{n+2}} < \lambda \le 0$, $w$ be the solution to (\ref{0315eq3.6}) with $w(0) = 1$ and $w'(0) = 0$. Then $w(\pi/2) < 0$ and $w'(\pi/2) < 0$.
\end{lemma}
\begin{proof}
 By making use of the substitution $\xi = \cos \phi$,  (\ref{0315eq3.6}) becomes into the following Legendre
  type differential equation:
\begin{equation}\label{0315eq3.7}
	(1 - \xi^2)\frac{{\rm d}^2w}{{\rm d}\xi^2} = n\,\xi\,\frac{{\rm d}w}{{\rm d}\xi} -A\,w
\end{equation}
with the initial conditions at $\xi= 1$:
\[
   w(1) = 1,\ \ \ \frac{{\rm d}w}{{\rm d}\xi}(1) = \frac{A}{n}.
\]
To prove the lemma, we need to show that $w = w(\xi)$ satisfies $w(0) < 0$ and $\frac{{\rm d}w}{{\rm d}\xi}(0) > 0$.
\\
Taking derivatives of (\ref{0315eq3.7}), we have the following second order differential equations:
\begin{gather}
	 	(1 - \xi^2)\frac{{\rm d}^3w}{{\rm d}\xi^3} = (n + 2)\,\xi\,\frac{{\rm d}^2w}{{\rm d}\xi^2} +(n - A)\,\frac{{\rm d}w}{{\rm d}\xi},  \label{0315eq3.8}\\
	 		(1 - \xi^2)\frac{{\rm d}^4w}{{\rm d}\xi^4} = (n + 4)\,\xi\,\frac{{\rm d}^3w}{{\rm d}\xi^3} + (2n+2 - A)\,\frac{{\rm d}^2w}{{\rm d}\xi^2}.  \label{0315eq3.9}
\end{gather}
It follows from (\ref{0315eq3.8}) and (\ref{0315eq3.9}) that
\[
    \frac{{\rm d}^2w}{{\rm d}\xi^2}(1) = - \frac{(n -A)A}{n(n + 2)},\ \ \ \frac{{\rm d}^3w}{{\rm d}\xi^3}(1) = \frac{(2n +2 -A)(n -A)A}{n(n + 2)(n + 4)}.
\]
By a direct calculation, we get that
$
 n - A = \frac{1}{2}\lambda(\sqrt{\lambda^2 + 4n} - \lambda) - n < 0
$
for $\lambda \le 0$, and we also get
$
   2n +2 - A = -\frac{1}{2}(\lambda^2 -4 -\lambda\sqrt{\lambda^2 + 4n})
$. Then one can obtain $\lambda > -\frac{2}{\sqrt{n + 2}}$ by solving inequality $ 2n +2 - A > 0$. Summing up, we have $ n - A <0$ and $  2n +2 - A > 0$ for $ -\frac{2}{\sqrt{n + 2}} < \lambda \le 0$, which tell us that
$\frac{{\rm d}^2w}{{\rm d}\xi^2}(1) > 0$ and $\frac{{\rm d}^3w}{{\rm d}\xi^3}(1) < 0$.  We also note that the differential equation (\ref{0315eq3.9}) for $ \frac{{\rm d}^2w}{{\rm d}\xi^2}$ satisfies a maximum principle since $2n +2 - A > 0$.
\\
Using the same method as in \cite{DK}, we can obtain $\frac{{\rm d}^2w}{{\rm d}\xi^2}(0) > 0$ and $\frac{{\rm d}^3w}{{\rm d}\xi^3}(0) < 0$, which follows from (\ref{0315eq3.7}) and (\ref{0315eq3.8}) that $w(0) < 0$ and $\frac{{\rm d}w}{{\rm d}\xi}(0) > 0$. Regarding $w = w(\phi)$ as a function of $\phi$, this says that $w(\pi/2) < 0$ and $\frac{{\rm d}w}{{\rm d}\xi}(\pi/2)<0$, which proves the lemma.
\end{proof}

\section{Proof of the theorem}

\noindent {\it Proof of Theorem \ref{0316thm1.1}.} Let $S[x]$ denote the solution to (\ref{0316eq2.5}) with initial conditions $S[x](0) = (x, 0, \pi/2)$.  According to the lemma \ref{0316lem2.1} we know that, for $x > -\lambda$, the first component of $P(S[x])$ written as a graph over the $r$-axis is concave down and decreasing before it blows-up at the point $B_x = (f_{x+\lambda}(r_*^{x+\lambda}),r_*^{x+\lambda}) = (x_*^{x+\lambda},r_*^{x+\lambda})$. The proposition \ref{prop3.1} tells us that when $x > -\lambda$ and $x$ closed to $-\lambda$, $B_x$ is in the first quadrant. From the lemma \ref{0317lem3.4} we know that when $  -\lambda< x < \frac{-\lambda\,+\,\sqrt{\lambda^2\,+\,4\,n}}{2} $ and $x$ closed to $\frac{-\lambda\,+\,\sqrt{\lambda^2\,+\,4\,n}}{2}$, $B_x$ is in the second quadrant. Since the mapping $(-\lambda, \infty) \rightarrow \mathbb{H}, x \mapsto B_x $ is continuous, the previous results imply that there exists $\hat{x} \in (-\lambda, \frac{-\lambda\,+\,\sqrt{\lambda^2\,+\,4\,n}}{2} )$ such that $B_{\hat{x}}$ lies in  $r$-axis. Suppose $P(S[\hat{x}](\hat{s})) = B_{\hat{x}}$, then the curve $[0, 2\hat{s}] \rightarrow \mathbb{H},\, s\mapsto P(S[\hat{x}](s))$ will generate  an embedded $\lambda$-hypersurface by (\ref{0316eq2.1}), which is diffeomorphic to a  sphere $\mathbb{S}^n$ and is not isometric to the sphere $S^n(\frac{-\lambda + \sqrt{\lambda^2 + 4n }}{2})$. Moreover, by (\ref{0426eq2.3}) and lemma \ref{0316lem2.1}, all the principal curvatures of the hypersurface we construct is positive. In other words, the hypersurface we construct is convex. \qed

\begin{remark}
According to some sketches, we know, for $n \ge 2$ and some $ \lambda \le -\frac{2}{\sqrt{n+2}} $,  there exists an  embedded $\lambda$-hypersurface $\Sigma^n \subset \R^{n+1}$ which is diffeomorphic to $\mathbb{S}^n$ and is not $S^n(\frac{-\lambda + \sqrt{\lambda^2 + 4n }}{2})$ (see Figure \ref{0321fig4.2}, Figure \ref{0321fig4.3}, Figure \ref{0426fig4.3} and Figure \ref{0426fig4.4}).
\end{remark}
\noindent
\begin{figure}[ht]
	\begin{minipage}[t]{0.5\linewidth}
		\centering
		\includegraphics[scale=0.16]{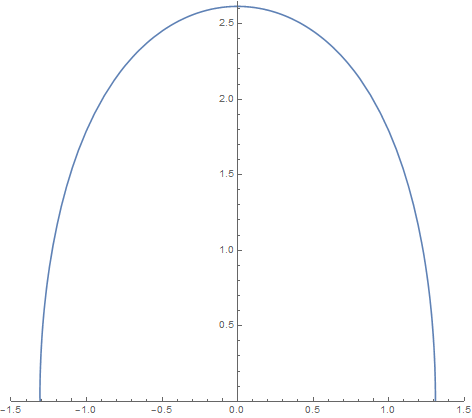}
		\caption{\,$n=2,\,\lambda=-1,\,\hat{x}\approx1.31$.}
		\label{0321fig4.2}
	\end{minipage}%
	\begin{minipage}[t]{0.5\linewidth}
		\centering
		\includegraphics[scale=0.14]{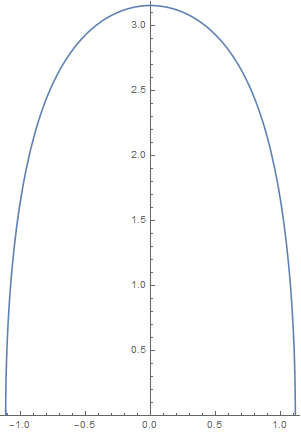}
		\caption{\,$n=3,\,\lambda=-0.9,\,\hat{x}\approx1.11$.}
		\label{0321fig4.3}
	\end{minipage}
\end{figure}

%
%
%

\noindent
\begin{figure}[ht]
		\begin{minipage}[t]{0.5\linewidth}
	\centering
	\includegraphics[scale=0.4]{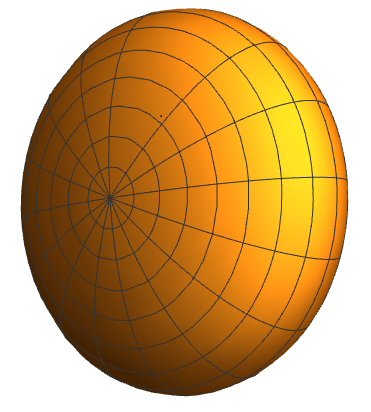}
	\caption{\ The graph of the $\lambda$-hypersurface generated by the profile curve in Figure 4.1}
	\label{0426fig4.3}
	\end{minipage}%
	\begin{minipage}[t]{0.5\linewidth}
	\centering
	\includegraphics[scale=0.4]{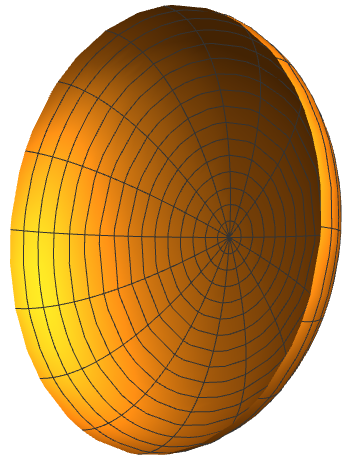}
	\caption{\ The graph of half of  the $\lambda$-hypersurface generated by the profile curve in Figure 4.1}
	\label{0426fig4.4}
	\end{minipage}
\end{figure}

\begin{remark}
In the proof of Theorem \ref{0316thm1.1}, we can not assure that the point $\hat{x}$ that we find in  $ (-\lambda, \frac{-\lambda\,+\,\sqrt{\lambda^2\,+\,4\,n}}{2} )$ is unique. According to some sketches, for some $n$ and $\lambda<0$ {\rm(}for example, $n = 3$ and $\lambda = -1$ {\rm )},
there may be more than one points $\hat{x}_1$, $\hat{x}_2,\cdots$ in $ (-\lambda, \frac{-\lambda\,+\,\sqrt{\lambda^2\,+\,4\,n}}{2} )$ so that $$P(S[\hat{x}_1](s)),P(S[\hat{x}_2](s)), \cdots$$ will generate embedded $\lambda$-hypersurfaces by (\ref{0316eq2.1}) (see Figure \ref{0321fig4.5} and Figure \ref{0321fig4.6}).
\end{remark}
\noindent
\begin{figure}[ht]
	\begin{minipage}[t]{0.5\linewidth}
		\centering
		\includegraphics[scale=0.18]{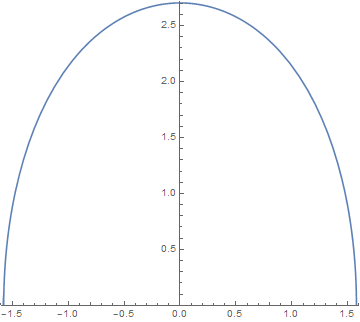}
		\caption{\ $n=3,\,\lambda=-1,\,\hat{x} \approx 1.57$.}
		\label{0321fig4.5}
	\end{minipage}%
	\begin{minipage}[t]{0.5\linewidth}
		\centering
		\includegraphics[scale=0.14]{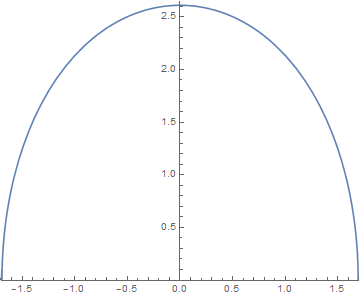}
		\caption{\ $n=3,\,\lambda=-1,\,\hat{x} \approx 1.69$.}
		\label{0321fig4.6}
	\end{minipage}
\end{figure}


\end{document}